\documentclass[11pt]{article}

\usepackage[mathscr]{eucal}
\usepackage{amsmath}
\usepackage{amssymb}
\usepackage{amsfonts}
\usepackage{amscd}
\usepackage{amsthm}
\usepackage{latexsym}
\usepackage{graphicx}

\def\be{\begin{equation}}
\def\ee{\end{equation}}
\def\bse{\begin{subequations}}
\def\ese{\end{subequations}}

\newtheorem{thm}{Theorem}

\newtheorem{lem}{Lemma}[section]

\def\bse{\begin{subequations}}
\def\ese{\end{subequations}}

\usepackage{geometry}

\def\XXint#1#2#3{{
\setbox0=\hbox{$#1{#2#3}{\int}$}
\vcenter{\hbox{$#2#3$}}\kern-.5\wd0}}

\title{Almost collapse mass quantization in 2D Smoluchowski-Poisson equation}

\author{Takashi Suzuki} 

\date{\today}

\begin{document}
\let\cleardoublepage\clearpage

\maketitle

\begin{abstract}
We study Smoluchowski-Poisson equation in two space dimensions provided with Dirichlet boundary condition for the Poisson part.  For this equation several profiles of blowup solution have been noticed.  Here we show collapse mass quantization with possible vanishing term. 
\end{abstract}

\section{Introduction}\label{sec:1}

We study parabolic-elliptic system composed of the Smoluchowski part 
\begin{equation}
u_t=\Delta u-\nabla \cdot u\nabla v \quad \mbox{in $\Omega\times(0,T)$} 
 \label{s}
\end{equation} 
with null-flux boundary condition 
\begin{equation}
\frac{\partial u}{\partial \nu}-u\frac{\partial v}{\partial\nu}=0 \quad \mbox{on $\partial\Omega \times (0,T)$} 
 \label{eqn:s3} 
\end{equation}
and the Poisson part in the form of 
\begin{equation}
-\Delta v=u, \quad \left. v\right\vert_{\partial\Omega}=0, 
 \label{eqn:s4} 
\end{equation}
where $\Omega\subset {\bf R}^2$ is a bounded domain with smooth boundary $\partial\Omega$ and $\nu$ is the outer unit normal vector.  Initial condition is given as 
\begin{equation} 
\left. u\right\vert_{t=0}=u_{0}(x)\geq 0 \quad \mbox{in $\Omega$}, 
 \label{eqn:s2}
\end{equation} 
where $u_0=u_0(x)$ is a smooth function.  

This model is proposed in statistical physics to describe the motion of mean field of many self-gravitating Brownian particles \cite{sc02}.  In a related case the Poisson part is provided with the Neumann boundary condition such as 
\begin{equation}
-\Delta v=u-\frac{1}{\vert\Omega\vert}\int_\Omega u, \quad \left. \frac{\partial v}{\partial \nu}\right\vert_{\partial\Omega}=0, \quad \int_\Omega v=0. 
 \label{eqn:sd} 
\end{equation}
Concerning (\ref{s})-(\ref{eqn:s3}), (\ref{eqn:s2}), and (\ref{eqn:sd}), there is a threshold of $\Vert u_0\Vert_1=\lambda$ for the blowup of the solution.  More precisely, if $\lambda<4\pi$ the solution exists global-in-time \cite{bil98, gz98, nsy97}.  If a local mass greater than $4\pi$ is concentrated on a boundary point, on the contrary, there arises blowup in finite time \cite{nag01, ss01b}. Underlying blowup mechanisms were also suspected from the study of stationary solutions \cite{cp81}. This attempt was followed by \cite{hv96, ss00}, using radially symmetric and general stationary solutions.  Up to now several properties have been known.  First, formation of collapses arises.  This means that the measure $u(x,t)dx$ is continuously extended up to the blowup time $t=T$ with its singular part composed of finite sum of delta functions \cite{ss01a}.  Next, quantization of the coefficients of these delta functions, called collapse masses, is assocaited with the formation of sub-collapses and type II blowup rate \cite{sen07, ns08}.  In this case total blowup mechanism is included in infinitesimally small parabolic region of space and time around $(x_0,T)$, where $x_0$ and $T$ denote the blowup point and time, respectively.  There is also a study on a multi-component system with chemotactic competitions  \cite{ess12}.  System (\ref{s})-(\ref{eqn:s2}), provided with Dirichlet condition for the Poisson part, is studied in \cite{s13} which excludes boundary blowup points.  

In contrast with (\ref{s})-(\ref{eqn:s3}), (\ref{eqn:s2}), and (\ref{eqn:sd}), the model (\ref{s})-(\ref{eqn:s2}) is hard in controlling the boundary blowup points, while its stationary state is equivalent to the mean field equation 
\begin{equation} 
-\Delta v=\frac{\lambda e^v}{\int_\Omega e^v}, \quad \left. v\right\vert_{\partial\Omega}=0 
 \label{eqn:mfe}
\end{equation}
where $\lambda=\Vert u_0\Vert_1$.  Equation (\ref{eqn:mfe}) arises also in statistical mechanics concerning point vortices \cite{ons49, jm73, pl76}, and the structure of the solution set has been studied in connection with the shape of $\Omega$ since \cite{ns90} (see, for example, \cite{gos11} and the references therein).  

To describe the connection between (\ref{s})-(\ref{eqn:s2}) and (\ref{eqn:mfe}), we confirm several fundamental features of the former.  First, local-in-time unique existence of the classical solution is standard, given smooth initial value $u_0=u_0(x)\geq 0$.  Henceforth, $T\in (0, +\infty]$ denotes its maximal existence time. If $u_0\not\equiv 0$, which we always assume below, the strong maximum principle and the Hopf lemma guarantee $u(\cdot,t)>0$ on $\overline{\Omega}$ for $t>0$.  Actually, system (\ref{s})-(\ref{eqn:s2}) is subject to thermodynamical laws, {\it total mass conservation} and {\it free energy decreasing}, 
\begin{eqnarray} 
& & \frac{d}{dt}\int_\Omega u=\int_\Omega \nabla\cdot (\nabla u-u\nabla v)=\int_{\partial\Omega}\frac{\partial u}{\partial \nu}-u\frac{\partial v}{\partial\nu} \ ds=0 
 \label{eqn:tmc} \\ 
& & \frac{d}{dt}{\cal F}(u)=-\int_\Omega u\vert \nabla (\log u-v)\vert^2\leq 0  
 \label{eqn:fed}
\end{eqnarray} 
where $ds$ denotes the surface element and 
\begin{equation} 
{\cal F}(u)=\int_\Omega u(\log u-1)-\frac{1}{2}\langle (-\Delta)^{-1}u,u\rangle 
 \label{eqn:free-e}
\end{equation} 
with $v=(-\Delta)^{-1}u$ standing for (\ref{eqn:s4}).  To see the reason why (\ref{eqn:mfe}) represents the stationary state of (\ref{s})-(\ref{eqn:s2}), we assume the vanishing of the right-hand side on (\ref{eqn:fed}). It follows that 
\begin{equation} 
\log u-v=\mbox{constant} 
 \label{eqn:unknown}
\end{equation} 
because $u>0$ everywhere.  This unknown constant is to be determined by (\ref{eqn:tmc}), that is, $\lambda=\Vert u\Vert_1$ prescribed in advance.  Consequently we obtain 
\[ u=\frac{\lambda e^v}{\int_\Omega e^v} \] 
which results in (\ref{eqn:fed}) from the Poisson part (\ref{eqn:s4}).  

Maximal existence time $T$ of non-stationary solution $u=u(\cdot,t)$, on the other hand, is estimated from below by $\Vert u_0\Vert_\infty$.  Hence $T<+\infty$ implies 
\[ \lim_{t\uparrow T}\Vert u(\cdot,t)\Vert_\infty=+\infty \] 
and the blowup set ${\cal S}$ defined by 
\begin{equation} 
{\cal S}=\{ x_0\in\overline{\Omega} \mid \mbox{$\exists x_k\rightarrow x_0$, $\exists t_k\uparrow T$ such that $u(x_k,t_k)\rightarrow+\infty$}\}  
 \label{eqn:blowupset1}
\end{equation} 
is not empty. In \cite{s13} we have studied this case to exclude boundary blowup.  Thus, if $T<+\infty$ in (\ref{s})-(\ref{eqn:s2}) then it holds that 
\begin{equation} 
u(x,t)dx\rightharpoonup \sum_{x_0\in {\cal S}}m(x_0)\delta_{x_0}(dx)+f(x)dx \quad \mbox{in ${\cal M}(\overline{\Omega})=C(\overline{\Omega})'$} 
 \label{eqn:bfq}
\end{equation} 
as $t\uparrow T$.  More precisely, the blowup set ${\cal S}$ defined by (\ref{eqn:blowupset1}) satisfies ${\cal S}\subset \Omega$ and $\sharp {\cal S}<+\infty$, and $0\leq f=f(x)\in L^1(\Omega)\cap C(\overline{\Omega}\setminus {\cal S})$.  The inner blowup mechanism, however, is more complicated than suspected by \cite{suzuki05, suzuki11}.  

Here we show the following theorem which refines the fact that collapse mass quantization implies type II blowup rate.  Henceforth, $C_i$, $i=1,2,\cdots, 25$, denote positive constants and 
\[ R(t)=(T-t)^{1/2}. \] 

\begin{thm}
Let $x_0\in {\cal S}$ and $t_k\uparrow T$.  Then there is a subsequence denoted by the same symbol and $m\in {\bf N}$, such that given $0<\varepsilon \ll 1$, we have $\tilde s>1$, $x_k^j\in B(x_0, C_1R(t_k'))$, and $0<b_j\leq \varepsilon$, $1\leq j\leq m$, satisfying 
\begin{eqnarray*} 
& & B(x_k^j, b_jR(t_k'))\cap B(x_k^i, b_iR(t_k'))=\emptyset, \qquad \qquad i\ne j, \ k \gg 1 \\ 
& & \limsup_{k\rightarrow\infty}\left\vert \Vert u(\cdot,t_k')\Vert_{L^1(B(x_k^j, b_jR(t_k'))}-8\pi\right\vert <\varepsilon, \quad \ 1\leq j\leq m \\ 
& & \lim_{b\uparrow+\infty}\limsup_{k\rightarrow\infty}R(t_k')\Vert u(\cdot, t_k')\Vert_{L^\infty(B(x_0, bR(t_k'))\setminus \bigcup_{j=1}^mB(x_k^j, b_jR(t_k')))}\leq \tilde s\varepsilon   
\end{eqnarray*} 
for $t_k'\uparrow T$ defined by $R(t_k')=\tilde s R(t_k)$. 
 \label{thm:2}
\end{thm} 

From the parabolic envelope (see (\ref{eqn:paraen1} below) we obtain  
\[ \lim_{b\uparrow+\infty}\lim_{t\uparrow T}\left\vert \Vert u(\cdot,t)\Vert_{L^1(B(x_0, bR(t))}-m(x_0)\right\vert=0. \] 
The above theorem, however, does not imply $m(x_0)\in 8\pi{\bf N}$.  For radially symmetric solution of ${\cal S}=\{0\}$ we have always $m(x_0)=8\pi$.  Formal solution with $m(x_0)=16\pi$ is also constructed in \cite{ssv13}.  The case $m(x_0)\not\in 8\pi {\bf N}$, however, has not yet been known. Key ingredients of the proof of the above theorem are weak scaling limit, concentration compactness principle, Liouville property, and improved $\varepsilon$ regularity. 

This paper is composed of three sections.  Taking preliminaries in section \ref{sec:2}, we show Theorem \ref{thm:2} in section \ref{sec:3}.

\section{Preliminaries}\label{sec:2}

We start with the {\it weak form} introduced by \cite{ss01a}.  It is derived from the symmetry of the Green's function $G(x,x')=G(x',x)$, taking the test fucntion in 
\begin{equation} 
{\cal X}=\{ \varphi\in C^2(\overline{\Omega}) \mid \mbox{$\left. \frac{\partial\varphi}{\partial \nu}\right\vert_{\partial\Omega}=0$} \}, 
 \label{eqn:x}
\end{equation} 
that is, 
\begin{equation} 
\frac{d}{dt}\int_\Omega \varphi u(\cdot,t)=\int_\Omega u(\cdot,t)\Delta\varphi+\frac{1}{2}\iint_{\Omega\times\Omega}\rho_\varphi(x,x')u(x,t)u(x', t) \ dxdx' 
 \label{eqn:prw}
\end{equation} 
for 
\[ \rho_\varphi(x,x')=\nabla\varphi(x)\cdot\nabla_xG(x,x')+\nabla\varphi(x')\cdot\nabla_{x'}G(x,x'). \] 
Equality (\ref{eqn:prw}) arises with 
\begin{equation} 
\rho_\varphi\in L^\infty(\Omega\times \Omega), 
 \label{eqn:double-finite}
\end{equation} 
derived from a delicate propeprty of $G(x,x')$.  More precisely, we have the interior regularity 
\begin{equation} 
G(x,x')=\Gamma(x-x')+K(x,x') 
 \label{eqn:ir} 
\end{equation}
with 
\begin{eqnarray*} 
& & K\in C^{2+\theta}((\Omega\times \overline{\Omega})\cup (\overline{\Omega}\times \Omega)), \quad 0<\theta<1 \\ 
& & \Gamma(x)=\frac{1}{2\pi}\log\frac{1}{\vert x\vert}, 
\end{eqnarray*} 
and also the boundary regularity 
\begin{equation}  
G(x,x')=E(x,x')+K(x,x') 
 \label{eqn:br}
\end{equation} 
with 
\begin{eqnarray*} 
& & K\in C^{2+\theta}(\overline{\Omega\cap B(x_0,R)}\times \overline{\Omega\cap B(x_0,R)}), \quad 0<\theta<1 \\ 
& & E(x,x')=\Gamma(X-X')-\Gamma(X-X_\ast') 
\end{eqnarray*} 
valid to $x_0\in\partial\Omega$ and $0<R\ll 1$.  In (\ref{eqn:br}) we use the conformal diffeomorphism 
\[ X:\overline{\Omega\cap B(x_0,2R)}\rightarrow \overline{{\bf R}^2_+}, \quad {\bf R}^2_+=\{ (\xi, \eta) \mid \eta>0 \} \] 
and also $X_\ast=(\xi, -\eta)$ for $X=(\xi, \eta)$.  With these relations we can confirm (\ref{eqn:double-finite}) for $\varphi\in {\cal X}$ (see \cite{s13}).  

The weak form (\ref{eqn:prw}) implies the monotonicity formula indicated by 
\begin{equation} 
\left\vert \frac{d}{dt}\langle \varphi, \mu(dx,t)\rangle\right\vert \leq C_{3}\Vert \nabla \varphi\Vert_{C^1(\overline{\Omega})} \quad \mbox{a.e. $t$} 
 \label{eqn:monotone}
\end{equation} 
where $\mu(dx,t)=u(x,t)dx$, $0\leq t<T$.  Since ${\cal X}$ is dense in $C(\overline{\Omega})$ and there holds the total mass conservation (\ref{eqn:tmc}), this $\mu(dx,t)$ is extended a $\ast$ weakly continuous measure up to $t=T$: $\mu(x,t)\in C_\ast([0,T], {\cal M}(\overline{\Omega})$.  Then $\epsilon$ regularity valid to (\ref{s})-(\ref{eqn:s4}) implies (\ref{eqn:bfq}) (see \cite{s13}).  More precisely, there is $\varepsilon_0>0$ such that 
\begin{equation} 
x_0\in \overline{\Omega}, \ \limsup_{t\uparrow T}\Vert u(\cdot,t)\Vert_{L^1(\Omega\cap B(x_0,R))}<\varepsilon_0, \ 0<R\ll 1 \quad \Rightarrow \quad x_0\not\in {\cal S}. 
 \label{eqn:e-reg}
\end{equation} 

Property (\ref{eqn:e-reg}) is a localization of the rigidness for the existence of the solution global-in-time, which means $T=+\infty$ for $0<\lambda\ll 1$ (see \cite{jl92} for (\ref{s})-(\ref{eqn:s3}), (\ref{eqn:s2}), and (\ref{eqn:sd})). It implies  
\[ \limsup_{t\uparrow T}\Vert u(\cdot,t)\Vert_{L^1(\Omega\cap B(x_0,R))}\geq \varepsilon_0, \quad 0<\forall R\ll 1, \ \forall x_0\not\in {\cal S}, \] 
while the above $\limsup_{t\uparrow T}$ is replaced by $\liminf_{t\uparrow T}$ because $u(x,t)dx$, $0\leq t<T$, is extended to $\mu(dx,t)\in C_\ast([0,T], {\cal M}(\overline{\Omega}))$: 
\[ \liminf_{t\uparrow T}\Vert u(\cdot,t)\Vert_{L^1(\Omega\cap B(x_0,R))}\geq \varepsilon_0, \quad  0<\forall R\ll 1, \ \forall x_0\not\in {\cal S}.  \] 
Then the total mass conservation (\ref{eqn:tmc}) implies the finiteness of blowup points, more precisely, 
\[ \sharp {\cal S}\leq \lambda/\varepsilon_0<+\infty. \] 
By the elliptic and parabolic regularities, the singular part of $\mu(dx,T)$, denoted by $\mu_s(dx,T)$, is composed of a finite sum of delta functions, 
\begin{equation} 
\mu(dx,T)=\sum_{x_0\in {\cal S}}m(x_0)\delta_{x_0}(dx)+f(x)dx 
 \label{eqn:wl}
\end{equation} 
with $m(x_0)\geq \varepsilon_0$ and $0\leq f=f(x)\in L^1(\Omega)\cap C(\overline{\Omega}\setminus {\cal S})$.

Here we note that this $\varepsilon$ regularity used in \cite{ss01a} is improved as follows.  
\begin{thm}[\cite{ss02b}]
We have $\varepsilon_0>0$ and $t_0\in (0,T)$ such that if the initial value $u_0=u_0(x)$ to (\ref{s})-(\ref{eqn:s4}) satisfies 
\begin{equation} 
\Vert u_0\Vert_{L^1(\Omega\cap B(x_0,R))}<2\varepsilon_0 
 \label{eqn:26}
\end{equation} 
for $x_0\in\overline{\Omega}$ and $0<R\ll 1$, then any $\tau\in (0,t_0)$ admits $C_4=C_4(\tau)$ such that 
\[ \sup_{t\in [\tau, t_0]}\Vert u(\cdot,t)\Vert_{L^\infty(\Omega\cap B(x_0, R/2))} \leq C_4, \] 
where $u=u(x,t)$ is the solution.  
 \label{thm:ss02b}
\end{thm}

The above $C_4$ is independent of any other quantities than (\ref{eqn:26}) concerning $u_0$, say, local $L^p$ norm in $p>1$. Actually, for the proof we use the parabolic regularity of such norms besides the monotonicity formula (\ref{eqn:monotone}).  Using scaling invariance of (\ref{s})-(\ref{eqn:s4}), Theorem \ref{thm:ss02b} takes the following form. 

\begin{lem}[improved $\varepsilon$-regularity]
Let $u=u(x,t)$ be a (classical) solution to (\ref{s})-(\ref{eqn:s4}) in $\Omega\times (-T,T)$ and let $u_0=u(\cdot,0)$.  Then there are positive constants $\varepsilon_0$, $\sigma$, and $C_3$ independent of $x_0\in \Omega$ and $0<R\ll 1$, such that 
\[ \Vert u_0\Vert_{L^1(B(x_0,R))}<\varepsilon_0 \] 
implies 
\[ \sup_{t\in [-\sigma_0R^2, \sigma_0R^2]}\Vert u(\cdot,t)\Vert_{L^\infty(B(x_0, R/2))}\leq C_5R^{-2}. \] 
 \label{lem:ier}
\end{lem}
\begin{proof} 
Since there is no blowup point of $u=u(x,t)$ on $\partial\Omega$, it keeps to be smooth near the boundary (see \cite{s13}). Hence there is a smooth subdomain $\omega\subset \subset \Omega$ such that 
\[ v(x,t)=v_0(x,t)+\int_\omega\Gamma(x-x')u(x',t)dx'+\int_\omega K(x,x')u(x',t)dx', \quad x\in \omega, \] 
where $v_0=v_0(x,t)$ is a smooth function on $\overline{\omega}\times [-T,T]$.  
Given $x_0\in\Omega$ and $0<R\ll 1$, we assume $x_0=0$ to take the scaling of $(u,v)$ as 
\begin{equation} 
u^\beta(x,t)=\beta^2u(\beta x, \beta^2t), \quad v^\beta(x,t)=v(\beta x, \beta^2t) 
 \label{eqn:si}
\end{equation} 
with $\beta=R$. 

This $(u^\beta, v^\beta)=(u^\beta(x,t), v^\beta(x,t))$ satisfies 
\begin{equation} 
u^\beta_t=\Delta u^\beta-\nabla\cdot u^\beta\nabla v^\beta, \quad -\Delta v^\beta=u^\beta \qquad \mbox{in $\beta^{-1}\Omega\times (-\beta^{-2}T, \beta^{-2}T)$}
 \label{eqn:193}
\end{equation} 
and also 
\begin{eqnarray*} 
& & v^\beta(x,t)=v_0(\beta x, \beta^2t)+\int_{\beta^{-1}\omega}(\Gamma(x-x')-\frac{1}{2\pi}\log\beta)u^\beta(x',t)dx' \\ 
& & \quad +\int_\omega K(\beta x, x')u(x',\beta^2t)dx'. 
\end{eqnarray*} 
Hence we obtain 
\[ \vert \nabla v^\beta(x,t)\vert \leq h_\beta(x,t)+C_6, \quad (x,t)\in \beta^{-1}\Omega\times (-\beta^{-2}T, \beta^{-2}T) \] 
with 
\[ h_\beta(x,t)=\frac{1}{2\pi}\int_{\beta^{-1}\omega}\frac{u^\beta(x',t)}{\vert x-x'\vert}dx', \quad \Vert u^\beta(\cdot,t)\Vert_{L^1(\beta^{-1}\Omega)}=\lambda. \] 

We fix $t$, $\beta$ for the moment, to take zero extension of $f(x)=u^\beta(x,t)/2\pi$ outside $\beta^{-1}\omega$.  Then it follows that 
\[ h_\beta(x,t)=\int_{{\bf R}^2}\frac{f(x')}{\vert x-x'\vert}dx'\leq \frac{\lambda}{2\pi}+\int_{\vert x-x'\vert\leq 1}\frac{f(x')}{\vert x-x'\vert}dx'=\frac{\lambda}{2\pi}+(g\ast f)(x) \] 
where $g(x)=\chi_B(x)/\vert x\vert$.  Since $g\in L^q({\bf R}^2)$, $1\leq q<2$, we obtain 
\[ \Vert h_\beta(\cdot,t)\Vert_{L^q(B)}\leq C_7=C_7(q), \quad 1\leq q<2 \] 
and hence 
\begin{equation} 
\sup_{t\in [-\beta^{-2}T, \beta^{-2}T]}\Vert \nabla v^\beta(\cdot, t) \Vert_{L^q(B)}\leq C_8=C_8(q), \quad 1\leq q<2. 
 \label{eqn:196}
\end{equation} 
Inequality (\ref{eqn:196}) implies also 
\begin{equation} 
\sup_{t\in [-\beta^{-2}T, \beta^{-2}T]}\Vert \tilde v^\beta(\cdot,t)\Vert_{L^p(B)}\leq C_9=C_9(p), \quad 1\leq p<\infty 
 \label{eqn:197}
\end{equation} 
for 
\[ \tilde v^\beta=v^\beta-\frac{1}{\vert B\vert}\int_Bv^\beta. \] 

Now we can repeat the proof of Theorem \ref{thm:ss02b} (see also Lemma 12.1 of \cite{suzuki05}), using (\ref{eqn:193}), (\ref{eqn:196}), and (\ref{eqn:197}) because $\tilde v^\beta$ may be replaced by $v^\beta$ in (\ref{eqn:193}).  Thus we have $\varepsilon_0$, $C_5$, and $\sigma_0$ such that 
\[ \Vert u_0^\beta\Vert_{L^1(B)}<\varepsilon_0 \quad \Rightarrow \quad \sup_{t\in [-\sigma_0, \sigma_0]}\Vert u^\beta(\cdot,t)\Vert_{L^\infty(B/2)}\leq C_5, \] 
which is equivalent to the assertion. 
\end{proof} 

{\it Weak solution} is introduced from the weak form (\ref{eqn:prw}) by \cite{ss02a}.  It is a fundamental tool in later arguments. Thus we say that $0\leq \mu=\mu(dx,t)\in C_\ast([0,T], {\cal M}(\overline{\Omega}))$ is a weak solution to (\ref{s})-(\ref{eqn:s4}) if there is 
\[ 0\leq \nu=\nu(\cdot,t)\in L^\infty_\ast(0,T; {\cal E}') \] 
called {\it multiplicate operator} satisfying the following properties, where ${\cal E}$ is the closure of the linear space 
\[ {\cal E}_0=\{ \psi+\rho_\varphi \mid \psi\in C(\overline{\Omega}\times\overline{\Omega}), \ \varphi\in {\cal X} \} \] 
in $L^\infty(\Omega\times\Omega)$: 
\begin{enumerate} 
\item For $\varphi\in {\cal X}$ the mapping $t\in [0,T]\mapsto \langle \varphi, \mu(dx,t)\rangle$ is absolutely continuous and there holds 
\begin{equation} 
\frac{d}{dt}\langle \varphi, \mu(dx,t)\rangle=\langle \Delta \varphi, \mu(dx,t) \rangle+\frac{1}{2}\langle \rho_\varphi, \nu(\cdot,t) \rangle_{{\cal E}, {\cal E}'} \quad \mbox{a.e. $t$}. 
 \label{eqn:mul1}
\end{equation} 
\item We have 
\begin{equation} 
\left. \nu(\cdot,t)\right\vert_{C(\overline{\Omega}\times\overline{\Omega})}=\mu(dx,t)\otimes \mu(dx',t) \quad \mbox{a.e. $t$}.  
 \label{eqn:mul2}
\end{equation} 
\end{enumerate} 
Here we confirm that the property $\nu\geq 0$ of $\nu\in {\cal E}'$ means 
\[ \left\vert \langle f, \nu\rangle_{{\cal E}, {\cal E}'}\right\vert\leq \langle g, \nu \rangle \] 
for any $f, g\in {\cal E}$ satisfying $\vert f\vert\leq g$ a.e. in $\Omega\times\Omega$.  

Total mass conservation of this weak solution is obvious, 
\[ \mu(\overline{\Omega},t)=\mu(\overline{\Omega}, 0), \quad t\in [0,T].  \] 
This weak solution, however, cannot be a measure-valued solution constructed in \cite{ds09, lsv12} (see also \cite{ssv13}) because of the following property.  

\begin{thm}[\cite{ss02a}] 
If the initial meause $\mu_0(dx)\in {\cal M}(\overline{\Omega})$ admits $x_0\in \Omega$ such that 
\[ \mu_0(\{ x_0\})>8\pi, \quad \lim_{R\downarrow 0}\frac{1}{R^2}\left\langle \vert x-x_0\vert^2\chi_{B(x_0,R)}, \mu_0(dx)\right\rangle =0 \] 
then there is no weak solution to (\ref{s})-(\ref{eqn:s4}) even local-in-time.  
 \label{thm:ss02a}
\end{thm} 

The second property is the generation of such a solution.  It follows because ${\cal E}$ is separable.  

\begin{lem}[\cite{ss02a}]
Let $\{ \mu_k(dx,t)\}_k\subset C_\ast([0,T], {\cal M}(\overline{\Omega}))$ be a sequence of weak solutions to (\ref{s})-(\ref{eqn:s4}). Let the associated multiplicate operator of $\mu_k(dx,t)$ be $\nu_k(\cdot,t)\in L^\infty_\ast(0,T; {\cal E}')$, and assume 
\begin{equation} 
\mu_{k}(\overline{\Omega}, 0)+\sup_{t\in [0,T]}\Vert \nu_k(\cdot,t)\Vert_{{\cal E}'} \leq C_{10}, \quad k=1,2,\cdots. 
 \label{eqn:w-bound}
\end{equation} 
Then there are $\mu(dx,t)\in C_\ast([0,T], {\cal M}(\overline{\Omega}))$ and $\nu(\cdot,t)\in L^\infty_\ast(0, T; {\cal E}')$ such that 
\begin{eqnarray*} 
& & \mu_k(dx,t)\rightharpoonup \mu(dx,t) \quad \mbox{in $C_\ast([0,T], {\cal M}(\overline{\Omega}))$} \\ 
& & \nu_k(\cdot,t)\rightharpoonup \nu(\cdot,t) \quad \qquad \mbox{in $L^\infty_\ast(0,T; {\cal E}')$}  
\end{eqnarray*} 
up to a subsequence, and this $\mu(dx,t)$ is a weak solution to (\ref{s})-(\ref{eqn:s4}) with the multiplicate operator $\nu(\cdot,t)$ satisfying 
\[ \mu(\overline{\Omega},0)+\Vert \nu(\cdot,t)\Vert_{{\cal E}'}\leq C_{10}.  \] 
 \label{pro:1}
\end{lem} 

Henceforth, we agree with the following notations. First, if $\mu(dx,t)$ has a density as 
\[ \mu(dx,t)=u(x,t)dx, \quad 0\leq u(\cdot,t)\in C([0,T], L^1(\Omega)), \] 
then the multiplicate operator is always taken as 
\[ \nu(\cdot,t)=u(x,t)u(x',t) \ dxdx', \] 
recalling ${\cal E}\subset L^\infty(\Omega\times\Omega)$.  Under this agreement, condition (\ref{eqn:w-bound}) is reduced to 
\begin{equation} 
\mu_{k}(\overline{\Omega},0)\leq C_{11}, \quad k=1,2,\cdots 
 \label{eqn:sk1} 
\end{equation}
if each $\mu_k(dx,t)$ takes density in $[0,T)$ such as 
\[ \mu_k(dx,t)=u_k(x,t)dx, \quad 0\leq u_k=u_k(\cdot,t)\in C([0,T), L^1(\Omega)). \] 
In fact, since we use 
\[ \nu_k(\cdot,t)=u_k(x,t)u_k(x',t)dxdx' \] 
in this case, inequality (\ref{eqn:sk1}) means 
\[ \Vert u_k(\cdot,t)\Vert_1=\Vert u_{k}(\cdot,0)\Vert_1=\mu_{k}(\overline{\Omega},0)\equiv \lambda_k \leq C_{11} \] 
and hence (\ref{eqn:w-bound}) follows with 
\[ \Vert \nu(\cdot,t)\Vert_{{\cal E}'}=\lambda_k^2\leq C_{11}^2. \] 

Next, we define the regularity of the above weak solution.  First, given $\mu=\mu(\cdot, t)\in {\cal M}(\overline{\Omega})$, we have a unique $v=v(\cdot,t)\in W^{1,q}(\Omega)$, $1\leq q<2$, such that 
\[ -\Delta v=\mu, \quad \left. v\right\vert_{\partial\Omega}=0. \] 
Let $I\subset (0,T)$ be an open interval and $\omega\subset \Omega$ an open set. If the weak solution $\mu(dx,t)$ has a density $u=u(\cdot,t)\in L^p(\omega)$ in $\omega\subset \Omega$, $1<p<\infty$, for $t\in I$, the above $v=v(\cdot,t)$ is in $W^{2,p}_{loc}(\omega)$ from the elliptic regularity. By Sobolev's and Morrey's imbedding theorems this implies $(u\nabla v)(\cdot,t)\in L^1_{loc}(\omega)$. Hence we can require the additional property 
\[ \frac{d}{dt}\langle \varphi, \mu(dx,t)\rangle=\langle \Delta \varphi(dx,t)\rangle+\langle \nabla \varphi\cdot\nabla v, \mu(dx,t)\rangle, \quad \mbox{a.e. $t\in I$} \] 
for any $\varphi\in C_0^2(\omega)$.  In such a case we say that $\mu(dx,t)$ is regular in $\omega\times I$. In Theorem \ref{pro:1}, if $\mu_k(dx,t)$ is regular with the density $u_k(x,t)$ in $\omega\times(0,T)$ satisfying 
\[ \sup_{t\in [0,T]}\Vert u_k(\cdot,t)\Vert_{L^p(\omega)}\leq C_{12} \] 
for $p>1$, then the generated $\mu(dx,t)$ is also regular in $\omega\times (0,T)$.  

Concluding this section, we turn to the {\it Liouville property} of this weak solution. Henceforth, we put ${\cal M}({\bf R}^2)=C_0({\bf R}^2)'$, where 
\[ C_0({\bf R}^2)=\{ f\in C({\bf R}^2\bigcup\{\infty\}) \mid f(\infty)=0 \} \] 
and ${\bf R}^2\bigcup \{\infty\}$ denotes one-point compactification of ${\bf R}^2$.  We can define the weak solution to 
\begin{equation} 
a_t=\Delta a-\nabla\cdot a\nabla\Gamma\ast a \quad \mbox{in ${\bf R}^2\times (-\infty, +\infty)$} 
 \label{eqn:entire}
\end{equation}
similarly (see the proof of Lemma \ref{lem:1.2} below for precise definition). In the following, $0\leq \varphi_{0,r}=\varphi_{0,r}(x)\leq 1$ denotes a smooth function with support contained on $\overline{B(0,r)}$ and equal to $1$ on $B(0,r/2)$.  

\begin{lem}[Liouville property]
Let $0\leq a=a(dx,t)\in C_\ast((-\infty,+\infty), {\cal M}({\bf R}^2))$ be a weak solution to (\ref{eqn:entire}) with uniformly bounded multiplicate operator. Then we have either $a({\bf R}^2,t)=8\pi$ or $a({\bf R}^2, t)=0$, exclusively in $t\in {\bf R}$. 
 \label{lem:1.2} 
\end{lem} 

\begin{proof} 
The estimate from above, $a({\bf R}^2, t)\leq 8\pi$, is done by \cite{suzuki11, s13}.  This property follows from (\ref{eqn:entire}) in ${\bf R}^2\times [0,+\infty)$.  Here we show the reverse part, $a({\bf R}^2,t)\geq 8\pi$ unless $a(dy,s)\equiv 0$, using (\ref{eqn:entire}) in ${\bf R}^2\times (-\infty,0]$.  

The proof, however, is similar.  In fact, by the definition, there is a multiplicate operator 
\[ 0\leq \kappa(\cdot,t)\in L^\infty_\ast(-\infty,0; {\cal K}'), \quad \Vert \kappa(\cdot,t)\Vert_{{\cal K}'}\leq C_{13} \] 
satisfying 
\begin{equation}
\frac{d}{dt}\langle \varphi, a(dx,t)\rangle=\langle \Delta\varphi, a(dx,t)\rangle+\frac{1}{2}\langle \rho_\varphi^0, \kappa(\cdot,t)\rangle_{{\cal K}, {\cal K}'} \quad \mbox{a.e. $t<0$} 
 \label{eqn:wf}
\end{equation} 
for each $\varphi\in C_0^2({\bf R}^2)$, where 
\[ \rho_\varphi^0(x,x')=-\frac{\nabla\varphi(x)-\nabla\varphi(x')}{2\pi\vert x-x'\vert^2}\cdot(x-x'). \] 
We put $\varphi(x)=\varphi_{0,1}(x/R)$ for $R>0$ in (\ref{eqn:wf}), integrate in $t$, and make $R\uparrow +\infty$ to conclude that $a({\bf R}^2,t)$ is constant in $t$, denoted by $m\geq 0$: 
\[ a({\bf R}^2, t)=m \quad -\infty<t\leq 0. \] 
Now we assume $m>0$ and derive $m\geq 8\pi$.  

First, we use {\it local second moment}, taking the smooth function $c=c(\alpha)$ defined on $\alpha\geq 0$ such that 
\begin{eqnarray} 
& & 0\leq c'(\alpha)\leq 1, \ -1\leq c(\alpha)\leq 0, \quad \alpha\geq 0 \nonumber\\ 
& & c(\alpha)=\left\{ \begin{array}{ll} 
\alpha-1, & 0\leq \alpha\leq 1/4 \\ 
0, & \alpha\geq 4. \end{array} \right. 
 \label{eqn:c}
\end{eqnarray} 
Then it holds that 
\begin{eqnarray*} 
& & \frac{d}{dt}\langle c(\vert x\vert^2)+1, a(dx,t)\rangle \\ 
& & \quad =\langle 4c''(\vert x\vert^2)\vert x\vert^2+4c'(\vert x\vert^2)-\frac{m}{2\pi}c'(\vert x\vert^2), a(dx,t)\rangle - \langle J, \kappa(\cdot,t) \rangle
\end{eqnarray*} 
with 
\[ J=J(x,x')=\frac{(c'(\vert x\vert^2)-c'(\vert x'\vert^2))(\vert x\vert^2-\vert x'\vert^2)}{4\pi\vert x-x'\vert^2}. \] 

Next we use 
\[ \vert J\vert \leq C_{14}(\varphi_{0,8}(x)+\varphi_{0,8}(x'))\{ (c(\vert x\vert^2)+1)+(c(\vert x'\vert^2)+1)\} \] 
and 
\[ \vert c''(\alpha)\alpha\vert \leq C_{15}(c(\alpha)+1) \] 
to deduce 
\begin{eqnarray*} 
& & \left\vert \frac{d}{dt}\langle c(\vert x\vert^2)+1, a(dx,t)\rangle-(4-\frac{m}{2\pi})\langle c'(\vert x\vert^2), a(dx,t)\rangle \right\vert \\ 
& & \quad \leq C_{16}\langle c(\vert x\vert^2)+1, a(dx,t)\rangle.  
\end{eqnarray*} 

Since 
\[ c(\alpha)+1+c'(\alpha)\geq \delta \] 
with $\delta>0$, it follows that 
\[ \frac{d}{dt}\langle c(\vert x\vert^2)+1, a(dx,t)\rangle \geq -C_{17}\langle c(\vert x\vert^2)+1, a(dx,t)\rangle+\delta m(4-\frac{m}{2\pi}) \] 
under the assumption of $0<m<8\pi$.  Therefore, if 
\[ \langle c(\vert x\vert^2)+1, a(dx,0)\rangle<\eta\equiv \delta m(4-\frac{m}{2\pi})/C_{17} \] 
is the case we obtain 
\[ \langle c(\vert x\vert^2)+1, a(dx,t)\rangle<0, \qquad t\ll -1, \] 
a contradiction.  Thus we have 
\[ 0<m<8\pi \quad \Rightarrow \quad \langle c(\vert x\vert^2)+1, a(dx,0)\rangle \geq \eta>0. \]  

Her we use the scaling invariance of (\ref{eqn:entire}) as in \cite{ko03}.  Namely, for each $\beta>0$ the measure 
\[ a^\beta(dx,t)=\beta^2a(dx',t'), \quad x'=\beta x, \quad t'=\beta^2t \] 
is again a weak solution satisfying $a^\beta({\bf R}^2,t)=m$.  Hence we obtain 
\[ \langle c(\vert x\vert^2)+1, a^\beta(dx,0)\rangle=\langle c(\beta^{-2}\vert x\vert^2)+1, a(dx. 0)\rangle \geq \eta \] 
if $0<m<8\pi$.  Letting $\beta\uparrow+\infty$, however, we get a contradiction $0\geq \eta$ by the dominated convergence theorem.  
\end{proof} 

\section{Proof of Theorem \ref{thm:2}}\label{sec:3}

Given $x_0\in{\cal S}$, we take the backward self-similar transformation 
\[ z(y,s)=(T-t)u(x,t), \quad y=(x-x_0)/(T-t)^{1/2}, \quad s=-\log(T-t). \] 
The underlying scaling invariance of (\ref{s})-(\ref{eqn:s4}) is (\ref{eqn:si}). Then we obtain 
\begin{eqnarray} 
& & z_s=\Delta z-\nabla\cdot z\nabla(w+\vert y\vert^2/4) 
 \label{eqn:bs1} \\ 
& & w(y,s)=\int_{\Omega_s}G_s(y,y')z(y',s) \ dy' \quad \mbox{in $\bigcup_{s>-\log T}\Omega_s\times \{s\}$} 
 \label{eqn:bs3}
\end{eqnarray} 
with 
\begin{equation} 
\frac{\partial z}{\partial \nu}-z\frac{\partial}{\partial\nu}(w+\vert y\vert^2/4)=0 \quad \mbox{on $\bigcup_{s>-\log T}\partial \Omega_s\times\{s\}$} 
 \label{eqn:bs2}
\end{equation} 
where $\Omega_s=(T-t)^{-1/2}(\Omega-\{x_0\})$ and $G_s(y,y')=G(x,x')$.  
Since 
\[ \int_{\Omega_s}z(y,s) \ dy=\int_\Omega u(x,t) \ dx=\lambda \] 
we can apply the argument used for the proof of Lemma \ref{pro:1} to the rescaled equation (\ref{eqn:bs1})-(\ref{eqn:bs2}).  

Thus given $t_k\uparrow+\infty$, we take 
\[ s_k=-\log (T-t_k) \ \uparrow+\infty. \] 
Passing to a subsequence denoted by the same symbol, we obtain 
\begin{equation} 
z(y, s+s_k)ds\rightharpoonup \zeta(dy,s) \quad \mbox{in $C_\ast(-\infty, +\infty; {\cal M}({\bf R}^2)$}). 
 \label{eqn:zeta}
\end{equation} 
Here zero extension is taken to $z(y,s+s_k)$ where it is not defined.  At this limiting process, inequality (\ref{eqn:monotone}) is used to derive the most important property, {\it parabolic envelope}, indicated by 
\begin{eqnarray} 
& & \zeta({\bf R}^2, s)=m(x_0)>0 
 \label{eqn:paraen1} \\ 
& & \langle \vert y\vert^2, \zeta(dy,s)\rangle \leq C_{18}
 \label{eqn:paraen2}
\end{eqnarray} 
valid to any $s\in (-\infty, +\infty)$ (see \cite{suzuki05, s13} for the proof). 

This $\zeta(dy,s)$ becomes a weak solution to (\ref{eqn:bs1})-(\ref{eqn:bs2}).  More precisely, if $x_0\in \Omega$, there is $0\leq \kappa=\kappa(\cdot,s)\in L^\infty_\ast(-\infty, +\infty; {\cal K}')$ with ${\cal K}$ the closure of the linear space 
\[ {\cal K}_0=\{ \psi+\rho_\varphi^0 \mid \psi \in C_0({\bf R}^2\times {\bf R}^2), \ \ \varphi\in C^2_0({\bf R}^2)\} \] 
in $L^\infty({\bf R}^2\times {\bf R}^2)$ and 
\[ \rho^0_\varphi(y,y')=-\frac{\nabla\varphi(y)-\nabla\varphi(y')}{2\pi\vert y-y'\vert^2}\cdot (y-y')  \] 
associated with $\Gamma(y)=\frac{1}{2\pi}\log\frac{1}{\vert y\vert}$.  This multiplicate operator $\kappa(\cdot,s)$ satisfies 
\begin{eqnarray*} 
& & \Vert \kappa(\cdot,s)\Vert_{{\cal K}'}\leq \lambda^2 \quad \mbox{a.e. $s$} \\ 
& & \left. \kappa(\cdot,s)\right\vert_{C_0({\bf R}^2\times{\bf R}^2)}=\zeta(dy,s)\otimes \zeta(dy',s) 
\end{eqnarray*}  
and there holds that 
\[ \frac{d}{ds}\langle \varphi, \zeta(dy,s)\rangle=\langle \Delta \varphi+\frac{y}{2}\cdot\nabla\varphi, \zeta(dy,s)\rangle+\frac{1}{2}\langle \rho^0_\varphi, \kappa(\cdot,s)\rangle \quad \mbox{a.e. $s$} \]  
for any $\varphi\in C_0^2({\bf R}^2)$, with the local absolute continuity of 
\[ s\in (-\infty, +\infty)\mapsto \langle \varphi, \zeta(dy,s)\rangle. \] 

If $x_0\in\partial\Omega$ the above $\zeta(dy,s)$ takes support on a closed half space, which we assume $\overline{{\bf R}^2_+}=\{ x=(\xi,\eta)\in {\bf R}^2\mid \eta\geq 0\}$ without loss of generality.  We obtain the same property as that of the above weak solution, replacing ${\cal K}_0$ and $\rho^0_\varphi$ by 
\[ {\cal K}_0=\{ \psi+\rho_\varphi^0 \mid \psi \in C_{0e}({\bf R}^2\times {\bf R}^2), \ \ \varphi\in C^2_0({\bf R}^2)\} \] 
and 
\[ \rho^0_\varphi(y,y')=\nabla\varphi(y)\cdot\nabla_yE(y,y')+\nabla\varphi(y')\cdot\nabla_{y'} E(y,y'),  \] 
respectively, where 
\[ E(y,y')=\Gamma(y-y')-\Gamma(y-y'_\ast) \] 
with $y_\ast=(\xi, -\eta)$ for $y=(\xi, \eta)$ and 
\begin{eqnarray*} 
& & C_{0e}=\{ f\in C_0 \mid f(\xi,-\eta)=f(\xi,\eta)\} \\ 
& & C_{0e}({\bf R}^2\times {\bf R}^2)=\{f\in C_0({\bf R}^2\times {\bf R}^2) \mid \\ 
& & \quad f(\xi_1, -\eta_1; \xi_2, \eta_2)=f(\xi_1, \eta_1: \xi_2, -\eta_2)=f(\xi_1, \eta_1; \xi_2, \eta_2) \}.  
\end{eqnarray*} 
We may call this $\zeta(dy,s)$ a weak solution to 
\begin{eqnarray*} 
& & z_s=\Delta z-\nabla\cdot z\nabla(w+\vert y\vert^2/4), \quad -\Delta w=z \qquad \mbox{in ${\bf R}^2_+\times (-\infty, +\infty)$} \\ 
& & \frac{\partial z}{\partial\nu}-z\frac{\partial}{\partial \nu}(w+\vert y\vert^2/4)=0 \qquad \qquad \qquad \qquad \quad \mbox{on $\partial {\bf R}^2_+\times(-\infty, +\infty)$} 
\end{eqnarray*} 
with uniformly bounded multiplicate operator. Such a solution provided with (\ref{eqn:paraen1})-(\ref{eqn:paraen2}), however, does not exist which excludes the boundary blowup of (\ref{s})-(\ref{eqn:s4}) (see \cite{s13}). More precisely, with 
\[ I(s)=\langle \vert y\vert^2, \zeta(dy,s)\rangle \] 
the formal calculation 
\begin{equation} 
\frac{dI}{ds}=4m(x_0)+I, \quad \mbox{a.e. $s$} 
 \label{eqn:second-b}
\end{equation} 
derived from 
\[\rho^0_{\vert y\vert^2}(y,y')=\nabla\vert y\vert^2\cdot\nabla_yE(y,y')+\nabla\vert y'\vert^2\cdot\nabla_{y'}E(y,y')=0 \] 
and 
\[ \frac{dI}{ds}=\langle \Delta \vert y\vert^2+\frac{y}{2}\cdot\nabla\vert y\vert^2, \zeta(dy,s)\rangle \] 
is justified by (\ref{eqn:paraen1})-(\ref{eqn:paraen2}).  Then (\ref{eqn:second-b}) with (\ref{eqn:paraen1}) implies 
\[ \lim_{s\uparrow+\infty}I(s)=+\infty, \] 
a contradiction to (\ref{eqn:paraen2}).  

In the case of $x_0\in \Omega$, we arrive at a weak solution $\zeta(dy,s)$ to 
\begin{equation} 
\zeta_s=\Delta \zeta-\nabla\cdot \zeta\nabla(\Gamma\ast \zeta+\vert y\vert^2/4) \quad \mbox{in ${\bf R}^2\times (-\infty, +\infty)$}.  
 \label{eqn:wsl1}
\end{equation} 
This time it follows that 
\[ \rho^0(y,y')=(\nabla\vert y\vert^2-\nabla \vert y'\vert^2)\cdot\nabla\Gamma(y-y') =-\frac{1}{\pi}, \] 
and then (\ref{eqn:second-b}) is replaced by 
\[ \frac{dI}{ds}=4m(x_0)-\frac{m(x_0)^2}{2\pi}+I \quad \mbox{a.e. $s$}. \] 
Properties (\ref{eqn:paraen1})-(\ref{eqn:paraen2}) now imply $m(x_0)\geq 8\pi$ with 
\begin{equation} 
I(s)=\frac{m(x_0)^2}{2\pi}-4m(x_0), \quad -\infty<s<+\infty
 \label{eqn:45}
\end{equation} 
and hence $m(x_0)\geq 8\pi$. 

Here we take the {\it scaling back} of $\zeta(dy,s)$, that is, the transformation 
\begin{equation} 
A(dy', s')=e^s\zeta(dy,s), \quad y'=e^{-s/2}y, \quad s'=-e^{-s}. 
 \label{eqn:c9}
\end{equation} 
This $0\leq A=A(dy, s)\in C_\ast((-\infty, 0], {\cal M}({\bf R}^2))$ becomes a weak solution to 
\begin{equation} 
A_s=\Delta A-\nabla\cdot A\nabla\Gamma\ast A \quad \mbox{in ${\bf R}^2\times (-\infty, 0)$} 
 \label{eqn:c7}
\end{equation}  
satisfying 
\begin{equation} 
A({\bf R}^2, s)=m(x_0), \quad -\infty<s\leq 0 
 \label{eqn:c8}
\end{equation} 
with a uniformly bounded multiplicate operator.  To use Lemma \ref{lem:1.2} now we take the translation limit.  

Thus, given $\tilde s_\ell\uparrow +\infty$, we take 
\begin{equation} 
A_\ell(dy)=A(dy,-\tilde s_\ell)/m(x_0) 
 \label{eqn:52}
\end{equation} 
to apply concentration compactness principle \cite{lions84} (see also p. 39 of \cite{struwe}).  There arises three alternatives, {\it compact}, {\it vanishing}, and {\it dichotomy}, passing to a subsequence.  

\begin{enumerate} 
\item ({\bf compact}) Each $0<\varepsilon<1$ admits $y_\ell\in {\bf R}^2$ and $R>0$ satisfying 
\begin{equation} 
A_\ell(B(y_\ell, R))>1-\varepsilon, \quad \forall \ell. 
 \label{eqn:204}
\end{equation} 
\item ({\bf vanishing}) It holds that 
\begin{equation} 
\lim_{\ell\rightarrow\infty}\sup_xA_\ell(B(x,R))=0
 \label{eqn:205}
\end{equation} 
for any $R>0$. 
\item ({\bf dichotomy}) There is $0<\lambda<1$ such that any $\varepsilon>0$ takes $y_\ell\in {\bf R}^2$ and $R>0$ such that 
\begin{eqnarray} 
& & \liminf_{\ell\rightarrow\infty}A_\ell(B(y_\ell, R))\geq \lambda-\varepsilon 
 \label{eqn:109} \\ 
& & \lim_{R'\uparrow+\infty}\liminf_{\ell\rightarrow\infty}A_\ell({\bf R}^2\setminus B(y_\ell, R'))\geq (1-\lambda)-\varepsilon. 
 \label{eqn:110}
\end{eqnarray} 
\end{enumerate}
We can apply Lemmas \ref{pro:1} and \ref{lem:ier} for the first and the second cases, respectively.  Then a hierarchical argument assures the following lemma. 

\begin{lem}[concentration compactness]
Given $t_k\uparrow T$, we put 
\[ s_k=-\log(T-t_k) \] 
to define the weak scaling limit $\zeta(dy,s)$ and its scaling back as in (\ref{eqn:zeta}) and (\ref{eqn:c9}), respectively.  Then we take $\tilde s_\ell\uparrow+\infty$ arbitrary, and define a family $\{ A_\ell(dy)\}_\ell$ of probability measures on ${\bf R}^2$ by (\ref{eqn:52}).  Then, passing to a subsequence we have $m\in {\bf N}\cup\{0\}$ such that any $\varepsilon>0$ admits $y_\ell^j\in {\bf R}^2$ and $b_j>0$, $1\leq j\leq m$, satisfying 
\begin{eqnarray} 
& & \lim_{\ell\rightarrow\infty}\vert y^i_\ell-y^j_\ell\vert=+\infty, \quad \forall i\neq j \label{eqn:208-1} \\ 
& & \limsup_{\ell\rightarrow \infty}\vert A_\ell (B(y^j_\ell, b_j))-8\pi\vert <\varepsilon, \quad \forall j \label{eqn:208-2} \\ 
& & \vert y^j_\ell\vert \leq C_{19}\tilde s_\ell^{1/2}, \quad \forall \ell\gg 1, \ \forall j. \label{eqn:211}
\end{eqnarray} 
Furthermore, there arises one of the following alternatives. 
\begin{enumerate}
\item $m(x_0)>8\pi m+\varepsilon$ and any $R>0$ admits $\ell_0$ such that 
\begin{eqnarray} 
& & \Vert A_\ell\Vert_{L^\infty({\bf R}^2\setminus \bigcup_{j=1}^mB(y^j_\ell, b_j))}\leq C_{20}R^{-2}, \quad \forall \ell \geq \ell_0 \label{eqn:209-1} \\ 
& & \liminf_{\ell\rightarrow\infty}A_\ell\left( {\bf R}^2\setminus \bigcup_{\ell=1}^mB(y_\ell^j, b_j)\right)\geq m(x_0)-8\pi m-\varepsilon. \label{eqn:209-2}
\end{eqnarray} 
\item $m(x_0)=8\pi m$ and 
\begin{equation} 
\limsup_{\ell\rightarrow\infty}A_\ell\left( {\bf R}^2\setminus \bigcup_{j=1}^mB(y_\ell^j, b_j)\right)<\varepsilon. 
 \label{eqn:210}
\end{equation} 
\end{enumerate} 
 \label{lem:3.1}
\end{lem} 
\begin{proof} 
If $\{A_\ell(dy)\}_\ell$ is compact, there holds that (\ref{eqn:204}).  Then we apply Lemma \ref{pro:1} to 
\[ a_\ell(dx,t)=A(dy,s), \quad x=y-y_\ell, \ t=s-\tilde s_\ell. \] 
Passing to a subsequence, we have the convergence 
\[ a_k(dx,t)\rightharpoonup a(dx,t) \quad \mbox{in $C_\ast(-\infty,+\infty; {\cal M}({\bf R}^2))$} \] 
and this $a=a(dx,t)$ is a weak solution to (\ref{eqn:entire}) satisfying 
\[ m(x_0)(1-\varepsilon)<a(B_{2R},0)\leq a({\bf R}^2,0)\leq m(x_0). \] 
Therefore, it holds that 
\[ m(x_0)=8\pi \] 
by Lemma \ref{lem:1.2}. We thus have (\ref{eqn:208-2}), (\ref{eqn:210}) for $m=1$, $b_1=R$, and $y_\ell^1=y_\ell$.  

If $\{ A_\ell(dy)\}$ is vanishing, there holds that (\ref{eqn:205}) for any $R>0$.  Hence, given $\varepsilon>0$, we obtain $\ell_0\gg 1$ such that 
\begin{equation} 
\sup_xA\left(B(x,R), -\tilde s_\ell\right)=\sup_{y}\zeta \left(B(y, \tilde s_\ell^{-1/2}R), -\log \tilde s_\ell\right)<\varepsilon/2, \quad \ell\geq \ell_0. 
 \label{eqn:214}
\end{equation} 
We fix such $\ell$, to put $c_0=\tilde s_\ell^{-1/2}R$.  Then, from (\ref{eqn:zeta}), any $b$ admits $k_0$ such that 
\begin{equation} 
\sup_{\vert y\vert\leq b}\Vert z(\cdot,-\log\tilde s_\ell+s_k)\Vert_{L^1(B(y,c_0))}<\varepsilon, \quad k\geq k_0. 
 \label{eqn:64}
\end{equation} 
For $\{ t_k'\}$ defined by 
\begin{equation} 
T-t_k'=\tilde s_\ell^{-1}(T-t_k) 
 \label{eqn:65}
\end{equation} 
inequality (\ref{eqn:64}) means 
\begin{equation} 
\sup_{x\in B(x_0, bR(t_k'))}\Vert u(\cdot,t_k')\Vert_{L^1(B(x, c_0R(t_k'))}<\varepsilon, \quad k\geq k_0, 
 \label{eqn:100}
\end{equation} 
recalling $R(t)=(T-t)^{1/2}$.  

From Lemma \ref{lem:ier}, (\ref{eqn:100}) implies 
\[ \sup_{\vert t-t_k'\vert \leq \sigma_0 c_0^2R(t_k')^2}\Vert u(\cdot,t)\Vert_{L^\infty(B(x, c_0R(t_k')/2)}\leq C_{21}c_0^{-2}R(t_k')^{-2}, \quad k\geq k_0  \] 
for any $x\in B(x_0,bR(t_k'))$.  Hence it follows that 
\begin{equation} 
\sup_{\vert t-t_k'\vert \leq \sigma_0 c_0^2R(t_k')^2}\Vert u(\cdot,t)\Vert_{L^\infty(B(x_0, bR(t_k'))}\leq C_{21}c_0^{-2}R(t_k')^{-2}, \quad k\geq k_0 
 \label{eqn:66}
\end{equation} 
In (\ref{eqn:66}) we have $\sigma_0c_0^2\leq 1/2$ with $\ell\gg 1$, which implies 
\[ \vert t-t_k'\vert \leq \sigma_0c_0^2R(t_k')^2 \quad \Rightarrow \quad \frac{1}{2} \leq \frac{R(t)^2}{R(t_k')^2}=1-\frac{t-t_k'}{T-t_k'} \leq \frac{3}{2}. \] 
Thus any $b>0$ admits $k_0$ such that 
\[ \sup_{\vert t-t_k'\vert\leq \sigma_0c_0^2R(t_k')^2}R(t)^2\Vert u(\cdot,t)\Vert_{L^\infty(B(x_0, bR(t))}\leq c_0^{-2}C_{22}, \quad k\geq k_0. \] 

Next, we put 
\[ s_k'=-\log\tilde s_\ell+s_k=-\log(T-t_k'). \] 
Since 
\[ \vert t-t_k'\vert \leq \sigma_0c_0^2R(t_k')^2 \quad \Leftrightarrow \quad \vert 1-e^{s_k'-s}\vert \leq \sigma_0c_0^2.  \] 
there is an absolute constant $s_0>0$ such that 
\begin{eqnarray} 
& & \sup_{\vert s-s_k'\vert\leq s_0c_0^2}\Vert z(\cdot,s)\Vert_{L^\infty(B_b)}= \nonumber\\ 
& & \quad \sup_{s\in [s_k-\log\tilde s_\ell-s_0c_0^2, s_k-\log \tilde s_\ell+s_0c_0^2]}\Vert z(\cdot,s)\Vert_{L^\infty(B_b)}\leq c_0^{-2}C_{19}, \quad k\geq k_0. 
 \label{eqn:218}
\end{eqnarray} 
Sending $k\rightarrow\infty$ and then $b\uparrow+\infty$, we obtain 
\begin{equation} 
\sup_{\vert s+\log\tilde s_\ell\vert\leq s_0R^2\tilde s_\ell^{-1}}\Vert \zeta(\cdot,s)\Vert_\infty \leq C_{22}R^{-2}\tilde s_\ell 
 \label{eqn:67}
\end{equation} 
by $c_0=R\tilde s_\ell^{-1/2}$.  Inequality (\ref{eqn:209-1}) with $m=0$, 
\begin{equation} 
\Vert A_\ell \Vert_\infty \leq C_{22}R^{-2}
 \label{eqn:220}
\end{equation} 
thus follows from (\ref{eqn:67}) for $s=-\log \tilde s_\ell$.  

In the third case of dichotomy, there is $0<\lambda_1<1$ such that any $\varepsilon>0$ takes $y_\ell^1\in {\bf R}^2$ and $R_1>0$ such that 
\begin{eqnarray} 
& & \liminf_{\ell\rightarrow\infty}A_\ell^1(B(y_\ell^1, R_1))\geq \lambda_1-\varepsilon \label{eqn:221-1} \\ 
& & \lim_{R'\uparrow+\infty}\liminf_{\ell\rightarrow\infty}A^1_\ell({\bf R}^2\setminus B(y_\ell^1, R'))\geq (1-\lambda_1)-\varepsilon, \label{eqn:221-2}
\end{eqnarray}
where $A_\ell^1=A_\ell$.  The main part of $A_\ell^1$ around $y_\ell^1$ is treated similarly to the compact case.  Then we obtain the first bubble centered at $y_\ell^1$ satisfying (\ref{eqn:208-2})-(\ref{eqn:211}) for $j=1$, together with $\lambda_1m(x_0)=8\pi$.  

To examine the residual part of $A_\ell^1$, we apply the concentration compactness principle to $\{ A_\ell^2(dy) \}_k$ defined by $A_\ell^2=\tilde A_\ell^2/m_\ell^2$, where 
\[ \tilde A_\ell^2=\left. A_\ell\right\vert_{B(y_\ell^1, R_1)^c}, \quad m_\ell^2=\tilde A_\ell^2({\bf R}^2). \] 
First, inequalities (\ref{eqn:109})-(\ref{eqn:110}) imply 
\begin{equation} 
\lim_{R'\uparrow+\infty}\limsup_{\ell\rightarrow\infty}A^2_\ell(B(y_\ell^1, R'))\leq 2\varepsilon. 
 \label{eqn:111}
\end{equation} 
If this $\{ A_\ell^2(dy)\}_\ell$ is compact, we have $y_\ell^2\in {\bf R}^2$ and $R_2>0$ satisfying 
\begin{equation} 
A_\ell^2(B(y_\ell^2, R_2))\geq 1-\varepsilon 
 \label{eqn:112}
\end{equation} 
up to a subsequence.  By (\ref{eqn:111}) and (\ref{eqn:112}) it holds that 
\[ \lim_{\ell\rightarrow\infty}\vert y_\ell^2-y_\ell^1\vert=+\infty. \] 
We take the translation limit smilarly to the first compact case, using Lemma \ref{lem:1.2}.  Then it holds that 
\[ \lim_{\ell\rightarrow\infty}m_\ell^2=m_\ast^2(\varepsilon)\in [m_2-2m(x_0)\varepsilon, m_2+2m(x_0)\varepsilon] \] 
for $m_2=m(x_0)-8\pi=(1-\lambda)m(x_0)$ together with 
\[ \lim_{\varepsilon\downarrow 0}m^2_\ast(\varepsilon)=8\pi. \] 
We thus end up with the collapse mass quantization, the second alternative with $m=2$. 

If $\{ A_\ell^2(dy)\}_\ell$ is vanishing, then Lemma \ref{lem:ier} is applicable.  We obtain the first alternative with $m=1$, similarliy.  In the rest case of dichotomy of $\{ A_\ell^2\}_\ell$, we proceed to the third process. Continuing this, we reach the alternatives.  

Finally, (\ref{eqn:45}) means $\langle \vert y\vert^2, A(dy,-\tilde s)\rangle = C_{23}\tilde s$ for $C_{23}=\frac{m(x_0)^2}{2\pi}-4m(x_0)$, which implies 
\[ A(B_R^c, -\tilde s)\leq C_{23}\tilde s/R^2. \] 
Then (\ref{eqn:211}) follows from (\ref{eqn:208-2}).  
\end{proof} 

Now we conclude the proof of Theorem \ref{thm:2}.  First, we assume (\ref{eqn:210}).  This condition implies that $\{ \tilde A_\ell^m\}_\ell$ defined by 
\[ \tilde A_\ell^m=\left. A_\ell\right\vert_{{\bf R}^2\setminus \bigcup_{j=1}^mB(y_\ell^j, b_j)}/A_\ell({\bf R}^2\setminus \bigcup_{j=1}^mB(y_\ell^j, b_j)) \] 
is vanishing.  Hence any $R>0$ admits $\ell_0$ such that (\ref{eqn:209-1}).  

This property means the existence of $\tilde s_1>0$ such that for any $\tilde s\geq \tilde s_1$ and $0<\varepsilon\ll 1$ we have $m(\tilde s)\in {\bf N}$, $y^j(\tilde s)\in {\bf R}^2$, and $b_j(\tilde s)>0$ for $1\leq j\leq m(\tilde s)$ satisfying  
\begin{eqnarray} 
& & \vert y^j(\tilde s)\vert \leq C_{24}\tilde s^{1/2}, \quad b_j(\tilde s)\leq C_{25}, \qquad 1\leq j\leq m(\tilde s) \label{eqn:72}\\ 
& & \lim_{\tilde s\uparrow+\infty}\inf_{1\leq i<j\leq m(\tilde s)}\vert y^i(\tilde s)-y^j(\tilde s)\vert=+\infty \label{eqn:73}\\ 
& & \limsup_{\tilde s\uparrow+\infty}\sum_{j=1}^{m(\tilde s)}\left\vert A(B(y^j(\tilde s), b_j(\tilde s)), -\tilde s)-8\pi\right\vert <\varepsilon \label{eqn:74}\\ 
& & \limsup_{\tilde s\uparrow+\infty}\Vert A(\cdot, -\tilde s)\Vert_{L^\infty({\bf R}^2\setminus \bigcup_{j=1}^{m(\tilde s)}B(y^j(\tilde s), b_j(\tilde s)))}=0. \label{eqn:75}
\end{eqnarray} 
By (\ref{eqn:74})-(\ref{eqn:75}), we may assume that $m(\tilde s)$ is independent of $\tilde s$, denoted by $m$. 

From the proof of Lemma \ref{lem:3.1}, therefore, each $0<\varepsilon\ll 1$ admits $\tilde s_1>0$, and $y^j(\tilde s)\in {\bf R}^2$, $b_j(\tilde s)>0$ for $1\leq j\leq m$, $\tilde s\geq \tilde s_1$ such that (\ref{eqn:72}), (\ref{eqn:73}) with $m(\tilde s)=m$ and  
\begin{eqnarray*} 
& & \lim_{b\uparrow+\infty}\limsup_{k\rightarrow +\infty}\Vert z(\cdot, -\log \tilde s+s_k)\Vert_{L^\infty(B_b\setminus \bigcup_{j=1}^m B(\tilde s^{-1/2}y^j(\tilde s), \tilde s^{-1/2}b_j(\tilde s))} \leq \varepsilon \tilde s \\ 
& & \limsup_{k\rightarrow\infty}\left\vert \Vert z(\cdot,-\log \tilde s+s_k)\Vert_{L^1(B(\tilde s^{-1/2}y^j(\tilde s), \tilde s^{-1/2}b_j(\tilde s))} -8\pi \right\vert<\varepsilon, \quad 1\leq j\leq m. 
\end{eqnarray*} 
Then Theorem \ref{thm:2} is obtained.

\begin{flushright}
Takashi Suzuki \\ 
Division of Mathematical Science \\ 
Department of Systems Innovation \\ 
Graduate School of Engineering Science \\ 
Osaka University \\ 
Machikaneyamacho 1-3 \\ 
Toyonakashi, 560-8531, Japan \\ 
suzuki@sigmath.es.osaka-u.ac.jp 
\end{flushright} 

\end{document}